\documentclass[oneside]{amsart}
\usepackage{amsmath}
\usepackage{amsfonts}
\usepackage{amssymb}
\usepackage{latexsym}
\usepackage{amsthm}
\usepackage{setspace}
\usepackage{hyperref}
\usepackage{geometry}
\usepackage[usenames,dvipsnames,svgnames,table]{xcolor}

\everymath{\displaystyle}
\usepackage{url}

\newcommand{\Oc}{\mathcal{O}}

\newcommand{\N}{\mathbb N}
\newcommand{\Z}{\mathbb Z}
\newcommand{\Q}{\mathbb Q}

\newcommand{\p}{\mathfrak p}

\newcommand{\Ac}{\mathcal{A}}

\DeclareMathOperator{\Hom}{Hom}

\DeclareMathOperator{\Nm}{N}
\DeclareMathOperator{\Tr}{tr}

\newcommand{\St}{\operatorname{\textrm{St}}}
\newcommand{\BC}{\operatorname{\textrm{BC}}}

\newtheorem{thm}{Theorem}

\newtheorem{prop}{Proposition}
\newtheorem*{result}{Result}

\theoremstyle{remark}

\begin{document}

\title[Genuine Bianchi modular forms of higher level]{Genuine Bianchi modular forms of higher level, \\at varying weight and discriminant}

\author{Alexander D. Rahm}
\address{Facult\'e des Sciences, de la Technologie et de la Communication, Universit\'e du Luxembourg}
\email{alexander.rahm@uni.lu}
\urladdr{http://math.uni.lu/~rahm/}

\author{Panagiotis Tsaknias}
\address{Facult\'e des Sciences, de la Technologie et de la Communication, Universit\'e du Luxembourg}
\email{panagiotis.tsaknias@gmail.com}
\urladdr{http://math.uni.lu/~tsaknias/}

\begin{abstract}
 Bianchi modular forms are automorphic forms over an imaginary quadratic field, 
 associated to a Bianchi group. 
 Even though modern studies of Bianchi modular forms go back to the mid 1960's, 
 most of the fundamental problems surrounding their theory are still wide open. 
 Only for certain types of Bianchi modular forms, which we will call non-genuine, 
 it is possible at present to develop dimension formulas: 
 They are (twists of) those forms which arise from elliptic cuspidal modular forms via the Langlands Base-Change procedure, 
 or arise from a quadratic extension of the imaginary quadratic field via automorphic induction (so-called CM-forms). 
 The remaining Bianchi modular forms are what we call genuine,
 and they are of interest for an extension of the modularity theorem (formerly the Taniyama-Shimura conjecture, crucial in the proof of Femat's Last Theorem) 
 to imaginary quadratic fields. 
 In a preceding paper by Rahm and \c{S}eng\"un, 
 an extreme paucity of genuine cuspidal Bianchi modular forms has been reported, 
 but those and other computations were restricted to level One. 
 In this paper, we are extending the formulas for the non-genuine Bianchi modular forms to deeper levels, and we are able to spot the first, rare instances of genuine forms at deeper level and heavier weight.
\end{abstract}

\maketitle

\section{Introduction}

Spaces of genuine Bianchi modular forms have been used Berger, Demb\'el\'e, Pacetti and \c{S}eng\"un
to construct evidence for the Brumer--Kramer paramodularity conjecture and the Eichler--Shimura conjecture~\cite{BergerDembelePacettiSengun}.
This evidence however consists so far in just one Abelian surface 
(that is paramodular in the Brumer--Kramer setting; 
and they deduce from it an Abelian surface satisfying the Eichler--Shimura conjecture).
No other non-trivial example for the Eichler--Shimura conjecture in dimension $2$ 
is present in the literature.
This is due to the limitations of the database of genuine weight $2$ Bianchi modular forms that was available to the four above-named authors.
That database~\cite{RahmSengun} did treat only level One Bianchi modular forms,
yielding extremely few genuine forms. 
So in order to systematically build up more evidence for the  Brumer--Kramer paramodularity conjecture and the Eichler--Shimura conjecture in dimension $2$,
it will not be enough to simply continue on a larger database of level One genuine Bianchi modular forms,
but rather one should use a database of higher level genuine Bianchi modular forms.
The latter type of database is provided with the present paper~\cite{database}.
Again, as for level One, the higher weight spaces of genuine forms are very rare,
but the weight $2$ spaces in our database are more abundant, and give hope for the construction of an algorithm that could systematically produce the desired Abelian surfaces.

In order to find the spaces of genuine Bianchi modular forms in the present database,
a major task was to establish dimension formulas for the non-genuine forms:
For level One, formulas were already established by Finis and Grunewald~\cite{FGT},
yielding the dimensions of Langlands Base-Change forms and Complex Multiplication (CM) forms outside the Base-Change space.
These dimensions of the non-genuine subspace did, in~\cite{RahmSengun},
only need to be subtracted from the full dimension of the Bianchi modular forms space,
which was computed on the machine from the geometry of the Bianchi modular group.
The latter machine computations extend to higher level via the Eckmann--Shapiro lemma,
which is implemented for calculating the cohomology that corresponds to the Bianchi modular forms space through the Eichler--Shimura(--Harder) isomorphism.
The Base-Change dimension formulas however did require a greater effort for their extension to higher level.
What added a particular additional challenge here, is that at higher levels, also twists of Base-Change appear, and have to be taken care of as part of the non-genuine space.

\begin{result}
Let $K$ be an imaginary quadratic field of discriminant $D_K$, odd class number $h_K$
and ring of integers $\mathcal{O}_K$.
Let $N\geq1$ be a square-free integer, coprime to $D_K$. 
Then we provide an explicit formula for the dimension of the space $S_k^{\textrm{nG}}(N\Oc_K)$
of non-genuine modular newforms of level $N\Oc_K$ and arbitrary weight $k$ in Theorem~\ref{thm:trivial}.
This formula comes with instructions (Section~\ref{Dimension formulas}) on how to evaluate it on a computer.
\end{result}

We also provide such formulas for class number $1$, 
$K=\Q(\sqrt{-p})$ for some prime $p\equiv 3\mod 4$ and level 
$\mathfrak{p}$ or $\mathfrak{p}^2$ where $\mathfrak{p}^2=(p)$
(see Section~\ref{level a power of the discriminant}).
But our dimension formulas need to be explicit enough to be evaluable on the machine, 
so they involve a lot of case distinctions,
and consequently are provided only for a part of the levels.
The authors hope that it should however become clear for experts how to obtain formulas for the missing levels.
For instance, for some theorems we assume that the class number is odd, 
which is equivalent to the discriminant being $-4$ or $-8$
or $-p$, where $p$ is a prime congruent to $3$ modulo $4$.
However, for even discriminants and class numbers, one should proceed analogously (see the remarks in Section~\ref{Final Remarks}).

In conclusion, we can say that the formulas and databases established with the present paper constitute a significant progress on the question raised in section 9.1 of~\cite{BergeronSengunVenkatesh}
on the exhaustion of spaces of newforms by non-genuine forms. 

\subsection*{Organization of the paper}
In Section \ref{sec:setting}, we fix the notation and assumptions that we will use throughout the paper.
We also derive a first quantitative expression for the dimension of the Base-Change space.
In Section~\ref{twists}, we take the twists of Base-Change into account.
In Section~\ref{sec:CM}, we establish formulas for the dimension of the space of CM-forms that are not Base-Change.
In Section~\ref{sec:total}, we provide a formula for the total dimension of the space generated by newforms that are non-genuine,
i.e. of any of the three types mentioned above.
This formula is made explicit in Section~\ref{Dimension formulas}, in a way that it can be evaluated on the machine.
We make some concluding remarks on our formulas in Section~\ref{Final Remarks}.
In Section~\ref{sec:computational}, we present our machine results, in which the 
non-genuine dimension is subtracted from the full dimension of the newforms space,
the latter being computed from the geometry of the Bianchi modular group.
Due to the sparsity of the genuine forms, those results can be considered as a treasure map to the conditions under which they exist.

\subsection*{Acknowledgements} 
This research project was funded by Gabor Wiese's FNR Luxembourg grant INTER/DFG/FNR/12/10/COMFGREP. 
We would like to thank John Cremona, Lassina Demb\'el\'e, Aurel Page and Gabor Wiese for helpful discussions.
Our very special thanks go to {M. Haluk \c{S}eng\"un}, 
for having initiated this research project, 
contributed helpful advice and 
having designed the MAGMA source code for computing the relevant cohomology of congruence subgroups via a version of the Eckmann-Shapiro lemma.

\newpage

\section{Setting}\label{sec:setting}

Let $D>1$ be a square-free integer and let $K=\Q(\sqrt{-D})$ and denote its ring of integers by $\mathcal{O}_K$. For any ideal $\mathfrak{n}$ of $\Oc_K$ and integer $k\geq2$ we let $S_{k,k}^K(\Gamma_0(\mathfrak{n}))$ denote the space of Cuspidal Bianchi modular forms over $K$ of weight $k,k$, level $\mathfrak{n}$ and trivial Nebentypus. This space admits an action of a Hecke algebra generated by operators indexed by the prime ideals of $\Oc_K$. 
Each eigenvector $F$ of this algebra corresponds to an automorphic representation 
$\Pi_F$ in $\mathcal{A}_2(K)$, the set of all cuspidal automorphic representations of GL$(2,\mathbb{A}_K)$.
We are interested in the ones appearing in the image of the Base-Change operator
defined by Langlands \cite{Langlands80},
$$\BC_\Q^K:\Ac_2(\Q)\to \Ac_2(K),$$
along with their twists. Here we generalise the formula of \cite{FGT} to arbitrary level $\mathfrak{n}$ coprime to the discriminant $D_K$ of $K$. More precisely, we provide an algorithm computing the dimension of this space without really computing the spaces in hand. In some simple cases we also provide explicit formulas. We also provide an algorithm for level $(p)=\mathfrak{p}^2$, where $p$ is a rational prime that ramifies in $K$. We hope that our proposed strategy outlines a complete plan for the interested reader to fill out the missing level cases without much effort.

It is clear that without loss of generality it is enough to study the Base-Change image problem on the new subspace of $S_{k,k}^K(\Gamma_0(\mathfrak{n}))$, which we will denote by $S_k(\mathfrak{n})$ in the rest of this paper.
A necessary condition for Base-Change to have non-trivial image is that the level is Galois stable, i.e. an ideal generated by a rational number. Let $\mathfrak{n}=N\Oc_K$ for some $N\geq 1$ coprime to~$D_K$. 
The first crucial observation is that one can consider the Base-Change operator locally. Let $\Pi=\BC_\Q^K(\pi)$ be an automorphic representation in the Base-Change image. 
Then for any prime ideal $\mathfrak{p}$ coprime to~$D_K$,
the conductor of $\Pi_\mathfrak{p}$ is the extension to $\Oc_K$ of the conductor of $\pi_p$, 
where $p$ is the prime below $\mathfrak{p}$. 
We therefore get that if $\pi$ contributes to the Base-Change subspace of 
$S_k(\mathfrak{n})$, then its conductor away from $D_K$ is exactly $N$.

This leaves us with the primes dividing the discriminant. Let $\mathfrak{p}$ divide $D_K$. Then $\Pi_\mathfrak{p}$ is a principal unramified series, since we required that the level is coprime to the discriminant. 
Let $\omega_p$ be the quadratic character of conductor $p$.
A complete list of all the possible local components $\pi_p$ of $\pi$ at $p$ is provided in \cite{FGT} which we also provide here for convenience:
\begin{itemize}
\item unramified Principal Series.
\item $1\oplus\omega_p$. 
\item A certain supercuspidal representation ( of conductor $p^2$ when $p>2$).
\end{itemize}

We now have a complete description of the inverse image of the Base-Change operator in terms of its inertial type at the primes $p|ND_K$: If $p|N$, then the type can be anything of conductor equal to the power of $p$ dividing $N$.
And if $p|D_K$, then it is one of the types mentioned above. 
In order to compute the dimension of the Base-Change subspace we would have to keep in mind that the Base-Change map is one-to-one when the component at a ramified prime is of the first kind and two-to-one if it is of the other two. 
Moreover, if a newform in the preimage spaces has CM by $K$, then its image is Eisenstein, so it does not contribute to the dimension. 
The above discussion yields the following formula.
\begin{prop} \label{prop:formula}
Let $S(d)$ be the set of prime divisors of $d$.
Let $\omega_d$ be the quadratic character of conductor $d$, under the convention that $\omega_d$ is trivial at a prime $p|d$ if and only if $p^2||d$. 
\\Then we obtain
$$\dim S_k^{\textrm{BC}}(N\Oc_K) =
\sum_{d|D^2_K} \frac{1}{2^{|S(d)|}}\bigg( \dim S_k^{d\textrm{-sc,new}}(\Gamma_0(Nd), \omega_d) - \dim S_k^{d\textrm{-sc,new}}(\Gamma_0(Nd), \omega_d)^{\textrm{CM}_K}\bigg),$$
where for each $d$, the space $S_k^{d-\textrm{sc,new}}(\Gamma_0(Nd), \omega_d)$ is the subspace of $S_k^{\textrm{new}}(\Gamma_0(Nd), \omega_d)$ 
spanned by newforms whose local type at every $p^2||d$ is the supercuspidal representation mentioned above.
And $S_k^{d\textrm{-sc,new}}(\Gamma_0(Nd), \omega_d)^{\textrm{CM}_K}$ is the $CM$ subspace.
\end{prop}
Of course many of the spaces appearing on the right hand side are trivial since the parity of their nebentypus does not match that of the weight $k$. 
The final key observation is that we have formulas for all of the spaces appearing on the right hand side of the above formula (see \cite{CO} or \cite{FGT}).

\section{Twists of Base-Change}\label{twists}

It is fairly easy to see that it is possible to obtain non Base-Change forms just by twisting a Base-Change one by a suitable non Base-Change character. We would like to count these forms out of our genuine forms subspace and therefore we would like to find the dimension of the subspace they span as well. 
Let us denote by $S_k^{\textrm{tBC}}(\mathfrak{n})$ the subspace of $S_k(\mathfrak{n})$ generated by those newforms that are twists of Base-Change. 
For this purpose, we say that $f \in S_k(\mathfrak{n})$ is \textit{twist of Base-Change},
if there exists a classical weight $k$ modular form $g$ of level $M \in \N$ with $M$ dividing $\mathfrak{n}$,
such that $f = \BC_\Q^K(g) \otimes c$ for some non-trivial character $c$.
We first prove the following:

% \begin{thm}\label{thm:twist-1}
% Let $K$ be an imaginary quadratic field of discriminant $D_K$ and class number $h_K=1$, and let $N>1$ be a square-free integer coprime to $D_K$. Set $\mathfrak{n}=N\Oc_K$. Then
% $\dim S_k^{\textrm{tBC}}(\mathfrak{n})=0$ for all $k\geq 2$.
% \end{thm}
%It should be clear from the proof that if a character turns a newform to be twist of Base-Change, then this character must be non Base-Change and unramified. 
%The existence of such characters is equivalent to the class number $h_K$ being greater than $1$. Moreover any such twist will be of the same level and weight. We thus get:
%
\begin{thm}\label{thm:twist-2}
Let $K$ be an imaginary quadratic field of discriminant $D_K$ and class number $h_K$.
Let $N\geq1$ be a square-free integer, coprime to $D_K$. Let $\mathfrak{n} = N \Oc_K$. 
Then for all $k\geq 2$, $$\dim S_k^{\textrm{tBC}}(\mathfrak{n})=(h_K -1) \dim S_k^{\textrm{BC}}(\mathfrak{n}).$$

\end{thm}
\begin{proof}
Since we are interested in newforms of square-free level, we get that the local type at any prime $\p$ is Steinberg or unramified principal series, depending on whether $\p|\mathfrak{n}$ or not. Let $\epsilon$ be a non Base-Change character of $\Oc_K^*$ and assume it has non-trivial conductor $c(\epsilon)$. Let $\p|c(\epsilon)$ be a prime ideal of $K$. Then $f\otimes\epsilon$ has level divisible by $\p^2$ and its type at $\p$ is $\epsilon\otimes \St$ or $\epsilon\oplus \epsilon$. In both cases, the type comes from Base-Change if and only if $\epsilon$ does. This shows that such an $\epsilon$ cannot be used to twist a non Base-Change form to a Base-Change one and vice versa. The only option left for $\epsilon$ is to be non Base-Change and to have trivial conductor. There precisely $h_k-1$ many such characters. Indeed each such character has the desired property which gives the desired result.
\end{proof}

\section{CM-Forms} \label{sec:CM}
Another subspace we would like to exclude as non genuine is the one generated by CM-newforms. They are the ones whose corresponding automorphic representation occurs as the automorphic induction of a suitable Hecke character of a quadratic extension $M/K$. Let us denote this subspace by $S_k^{\textrm{CM}}(\mathfrak{n})$. It is quite often the case that $S_k^{\textrm{CM}}(\mathfrak{n})\subseteq S_k^{\textrm{BC}}(\mathfrak{n})$, but not always. Nevertheless one can easily prove the following:

\begin{thm}\label{thm:CM}
Let $K$ and $\mathfrak{n}$ be as in Theorem \ref{thm:twist-2}. Moreover assume that $h_K$ is odd. Then for all $k\geq 2$,
$$\dim S_k^{\textrm{CM}}(\mathfrak{n}) = 0.$$ 
\end{thm}

\begin{proof}
Let $f$ be a newform that supposedly lies in $S_k^{\textrm{CM}}(\mathfrak{n})$. Let $\psi$ be the Hecke character of some extension $M/K$ whose automorphic induction is $f$. We then have that $L(f,s)=L(\psi,s)$. Equating local components at any prime $\p$ dividing the discriminant $\mathfrak{d}_{M/K}$ we see that $\p^2$ should divide~$\mathfrak{n}$. 
Since $\mathfrak{n}$ is assumed to be square-free no such $\p$ should exist and thus the extension $M/K$ is unramified. 
This provides the desired contradiction since the class number of $K$ is assumed to be odd.
\end{proof}

%The theorem is actually true even when the class number is odd since even in this case $K$ admits no unramified quadratic extension. We thus have that:
%
%\begin{thm} 
%Let $K$ and $\mathfrak{n}$ be as in Theorem \ref{thm:twist-2}. Moreover assume that $h_K$ is odd. Then $$\dim S_k^{\textrm{CM}}(\mathfrak{n}) = 0,$$ for all $k\geq 2$. 
%\end{thm}

\section{Total non-genuine subspace} \label{sec:total}

In this section we provide a formula for the total dimension of the space generated by newforms that are "non-genuine" in any of the three notions discussed above, i.e. a formula for
$$\dim \big(S_k^{\textrm{BC}}(\mathfrak{n}) + S_k^{\textrm{tBC}}(\mathfrak{n})+S_k^{\textrm{CM}}(\mathfrak{n})\big).$$
Let us denote this space by $S_k^{\textrm{nG}}(\mathfrak{n})$. Then combining Theorem~\ref{thm:twist-2}, Theorem~\ref{thm:CM} and Proposition~\ref{prop:formula} yields the following:
\begin{thm}\label{thm:trivial}
Let $K$ be an imaginary quadratic field of discriminant $D_K$ and odd class number $h_K$, and $N\geq1$ be a square-free integer, coprime to $D_K$. 
 Let $\mathfrak{n} = N \Oc_K$. Then
\begin{equation}\label{eq:main}
  \begin{split}
\dim S_k^{\textrm{nG}}(\mathfrak{n}) &= h_K \dim S_k^{\textrm{BC}}(\mathfrak{n})\\
  &= h_K\sum_{d|D^2_K}  \frac{1}{2^{|S(d)|}}\bigg(\dim S_k^{d\textrm{-sc,new}}(\Gamma_0(Nd), \omega_d) - \dim S_k^{d\textrm{-sc,new}}(\Gamma_0(Nd), \omega_d)^{\textrm{CM}_K}\bigg).
  \end{split}
\end{equation}
\end{thm}

\section{Dimension formulas} \label{Dimension formulas}

In this section we provide explicit formulas/algorithm for the right hand side of the formula in Theorem \ref{thm:trivial}. We split the computation into two parts, each involving one of the two main ingredients of each summand in the right hand side of (\ref{eq:main}). 

\subsection{Computing $\dim S_k^{d\textrm{-sc,new}}(\Gamma_0(Nd), \omega_d)$}

Our starting point is Proposition 4.18 in \cite{FGT} which we provide here for convenience.
Here, we denote by $S_k(\Gamma(N))$ the space of weight $k$ elliptic modular forms
for the principal congruence subgroup $\Gamma(N) \subseteq$ SL$(2,\Z)$ of level $N$.
\begin{prop}[Finis, Grunewald, Tirao]\label{prop:key}
Let $N \geq 1$ and $k \geq 2$ be integers and $\sigma$ a representation of $G_N = {\rm SL}(2,\Z/N\Z)$
such that $\sigma(-I_2)$ is the scalar $(-1)^k$. Let $U_N \subseteq G_N$ be the subgroup of all upper triangular unipotent elements and $S_3$ and $S_4$ the images in $G_N$ of elements of
${\rm SL}(2,\Z)$ of order $3$ and $4$, respectively. Then,
$$\dim \Hom_{G_N}(S_k(\Gamma(N)), \sigma) = \frac{k-1}{12}\dim\sigma - \frac{1}{2}\dim\sigma^{U_N} + \epsilon_k\Tr\sigma(S_3) + \mu_k \Tr\sigma(S_4) + \delta_{k,2}\dim \sigma^{G_N}.$$
\end{prop}
The constants $\epsilon_k$ and $\mu_k$ are explicit functions of $k$ and $\delta_{k,2}$ is the usual Kronecker delta notation.
In what follows, given a subspace of newforms $B\subseteq S_k(\Gamma(N))$, we will say that $\sigma$ \emph{defines} $B$ if $\Hom_{G_N}(S_k(\Gamma(N)), \sigma) \cong B$. We thus need to identify a suitable $\sigma$ for each subspace involved in the last line of (\ref{eq:main}). In fact we only need to compute the following five invariants associated to such
a sigma:
$$\dim \sigma,\ \dim\sigma^{U_N},\ \Tr\sigma(S_3),\ \Tr\sigma(S_4),\ \dim \sigma^{G_N}.$$
We will denote them by $I_i(\sigma)$, $i\in\{1,2,3,4,5\}$ respectively. It is important to notice that $\sigma$, and therefore the five invariants associated with it, depends only on the level structure of the
subspace that is of interest to us and not the weight. It is also clear that the following properties hold:
$$I_i(\sigma\oplus\sigma')=I_i(\sigma) + I_i(\sigma')$$
and
$$I_i(\otimes_p\sigma_p)=\prod_pI_i(\sigma_p)$$
for all $\sigma=\otimes_p\sigma_p$, $\sigma'$ and for all $i\in\{1,2,3,4,5\}$.

Let us fix one of the spaces in the right hand side of the formula in Theorem \ref{thm:trivial}, $S_k^{d\textrm{-sc,new}}(\Gamma_0(Nd), \omega_d)$ say, and let $\sigma$ be the representation defining it. In view of the properties just mentioned, we will compute the $I_i(\sigma)$ by determining the $N$-part and the $d$-part separately.

In order to compute the $N$-part, it is enough to notice that it corresponds to the one defining $S_k^{\textrm{new}}(\Gamma_0(N))$. Given an integer $N\geq1$, let $\sigma^{N}=\otimes_p \sigma^{N}_p$ and $\sigma^{N,\textrm{new}}=\otimes_p \sigma^{N,\textrm{new}}_p$ be the representation of $G_N$ defining $S_k(\Gamma_0(N))$ and $S_k^{\textrm{new}}(\Gamma_0(N))$ respectively. It is immediate then that $\sigma^{N}_p=\sigma^{p^e}$ and $\sigma^{N,\textrm{new}}_p=\sigma^{p^e,\textrm{new}}$, where $p^e||N$. Moreover it is easy to see that
$$\dim S_k^{\textrm{new}}(\Gamma_0(p^e)) = \dim S_k(\Gamma_0(p^e)) - 2\dim S_k(\Gamma_0(p^{e-1})) + \dim S_k(\Gamma_0(p^{e-2})),$$
for any $e\geq 0$, with the understanding that the dimensions mentioned are $0$ if the exponent of $p$ becomes negative. This in turn implies the following formula:
\begin{equation}\label{eq:newsubdim}
I_i(\sigma^{p^e,\textrm{new}}) = I_i(\sigma^{p^e}) - 2I_i(\sigma^{p^{e-1}}) + I_i(\sigma^{p^{e-2}}).
\end{equation}
As before, the $I_i$'s involved are $0$ if the corresponding exponent of $p$ is negative. Formulas for the right hand side are provided in \cite{CO} and we also state them here for convenience:
\begin{align*}
I_1(\sigma^{p^e}) & = 	\begin{cases}
					1						& e=0\\
					p^{e-1}(p+1) 				& e\geq 1\\
					\end{cases}\\
I_2(\sigma^{p^e}) & = 	\lambda(e,0,p) = \begin{cases}
					1						&e=0\\
					2p^n 					&e=2n+1\\
					p^n+p^{n-1}				&e=2n\geq2\\
					\end{cases}\\
I_3(\sigma^{p^e}) & = 	\#\{x \mod p^e| x^2+x+1=0\}=\begin{cases}
					1						&e=0\textrm{ or }p^e=3\\
					1 + \Big(\frac{-3}{p}\Big)	&e\geq 1\textrm{ and }p\neq3\\
					0						&e\geq2\textrm{ and }p=3\\
					\end{cases}\\
I_4(\sigma^{p^e}) & = 	\#\{x \mod p^e| x^2+1=0\}=\begin{cases}
					1						& e=0\textrm{ or }p^e=2\\
					1 + \Big(\frac{-1}{p}\Big)	&e\geq1\textrm{ and } p\neq 2\\
					0						&e>1\textrm{ and } p=2\\
					\end{cases}\\
I_5(\sigma^{p^e}) & = 	1
\end{align*}
where $\lambda(sr_p, s_p, p)$ is the one defined in \cite{CO}. Using (\ref{eq:newsubdim}) and the above one easily gets:
\begin{align}
I_1(\sigma^{p^e,\textrm{new}})  &= 	\begin{cases}
								1 					& e= 0\\
								p-1 					& e = 1\\
								p^2-p-1 				& e = 2\\
								p^{e-3}(p-1)^2(p+1) 	& e \geq 3\\
								\end{cases}\label{eq:I1new}\\
I_2(\sigma^{p^e,\textrm{new}}) & = 	\begin{cases}
								1 					& e= 0\\
								0					& e=2n+1\\
								p-2					& e= 2\\
								p^{n-2}(p-1)^2		& e= 2n\geq 4\\
								\end{cases}\\
I_3(\sigma^{p^e,\textrm{new}})  &= 	\begin{cases}
								1 					& e= 0\textrm{ or }p^e=3^3\\
								\Big(\frac{-3}{p}\Big)-1	& e=1\textrm{ and }p\neq3\\
								-\Big(\frac{-3}{p}\Big)	& e=2\textrm{ and }p\neq3\\
								-1					&p^e=3\textrm{ or }3^2\\
								0 					&\textrm{otherwise}\\
								\end{cases}
\end{align}
\begin{align}
I_4(\sigma^{p^e,\textrm{new}})  &= 	\begin{cases}
								1 					& e= 0\textrm{ or }p^e=2^3\\
								\Big(\frac{-1}{p}\Big)-1	& e=1\textrm{ and }p\neq2\\
								-\Big(\frac{-1}{p}\Big)	& e=2\textrm{ and }p\neq2\\
								-1					&p^e=2\textrm{ or }2^2\\
								0 					&\textrm{otherwise}\\
								\end{cases}\\
I_5(\sigma^{p^e,\textrm{new}})  &= 	\begin{cases}
								1 					& e= 0\\
								-1					& e=1\\
								0 					& e\geq 2\\
								\end{cases}\label{eq:I5new}
\end{align}
Let $\tau^d$, where $d|D_K^2$, be the $d$-part of the representation $D_K$. The $I_i$'s for each $\tau^d_\ell$, with $\ell|d$, have already been determined in \cite{FGT}. We then have that $\sigma = \sigma^N\otimes \tau^d$ defines $S_k(\Gamma_0(Nd), \omega_d)$. Moreover $I_i(\sigma^N\otimes \tau^d) = I_i(\sigma^N)I_i(\tau^d)$ and we have explicit formulas for both terms in the right hand side. 
$$\dim S_k(\Gamma_0(Nd), \omega_d) = \frac{k-1}{12}I_1(\sigma^N)I_1(\tau^d)+\cdots.$$
If one lets $I_i(\sigma_p)=1$ for all $i$ and all $p\nmid D_K$ then one gets the formulas derived in \cite{FGT}.

\subsection{Computing $\dim S_k^{d\textrm{-sc,new}}(\Gamma_0(Nd), \omega_d)^{\textrm{CM}_K}$}
We follow the method used in \cite{Tsaknias2012a} to count CM newforms of a given level, weight and Nebentypus. As explained there, every newform has associated to it a pair of characters $(\psi_f, \psi_\infty)$ that satisfies a certain compatibility condition. Moreover, for every such compatible pair, there are $h_K$ many newforms associated with it. The problem is thus reduced to counting compatible pairs that match the prescribed level, weight and Nebentypus. The infinity component is always uniquely determined by the weight, so we only have to count all the $\psi_f$'s that correspond to the given level and Nebentypus and are compatible with the given weight. Finally every such $\psi_f$ is determined by its local components at every prime $\p$ of $\Oc_K$, so we count the possible choices for each of those and multiply them to get the final answer.

Recall that $Nd=\Nm(\mathfrak{m}) |D_K|$, where $\mathfrak{m}$ is the conductor of $\psi_f$. Since $N$ and $D_K$ are coprime, every prime dividing $N$ should divide $\Nm(\mathfrak{m})$ too. Moreover, a prime $\ell|d$ divides $\Nm(\mathfrak{m})$ if and only if $\ell^2|D_K$. We will treat primes $\p$ separately, depending on their ramification and residue class degree, $e_\p$ and $f_\p$ respectively. We start with primes $p$ whose residue characteristic is greater than $2$.

\subsubsection{p inert} In this case $p=\p$ and $\Nm(\p)=p^2$, so the exponent of $p$ in $N$ should be even. Assume that this is the case: $p^{2n}||N$, for some $n\geq1$. We then get that the conductor of $\psi_f$ has to be divisible exactly by  $\p^n$. As we can see in \cite{Ranum1910}, $(\Oc/\p^n)^*$ is generated by three independent generators: $\xi$, $1+p$ and $1+p\omega$ of order $p^2-1$, $p^{n-1}$ and $p^{n-1}$ respectively. The residues of rational integers form a subgroup isomorphic to $(\Z/p^n\Z)^*$ which is generated by $\xi^{p+1}$ and $1+p$.Since the restriction of $\psi_f$ on $\Z^*$ is predetermined by the nebentypus
we have unique choices for $\xi^{p+1}$ and $1+p$. In our case the nebentypus is trivial and therefore $\psi_f(\xi^{p+1})=1$ and $\psi_f(1+p)=1$. If $n=1$, then the only generator to consider is $\xi$ and since we need the conductor at $\p$ to be $\p$ we have to exclude $1$ from the possible values of $\psi_f(\xi)$, which leaves $p$ choices in total. If however $n>1$ then in order for the conductor at $\p$ to be $\p^n$ we need either $\psi_f(1+p)$ or $\psi_f(1+p\omega)$ to be a primitive $p^{n-1}$-th root of  unity. Since $\psi_f(1+p)=1$ we get that $\psi_f(1+p\omega)$ must satisfy this condition which leaves $\varphi(p^{n-1})=p^{n-2}(p-1)$ choices. In this case all $p+1$ choices for $\psi_f(\xi)$ 
are permitted, so we get in total $p^{n-2}(p^2-1)$ many choices.

\subsubsection{$p$ split} In this case $p=\p\bar\p$ and $\Nm(\p)=\Nm(\bar\p)=p$. Let $p^t||N$ for some $t\geq1$. Then the $p$-part of the conductor of $\psi_f$ is apriori of the form $\p^\alpha\bar\p^\beta$ for any $0\leq\alpha,\beta\leq t$ such that $\alpha+\beta=t$. We will first show that in our case $\alpha=\beta=n\geq1$ and therefore $t$ must be even. We have the following group homomorphisms:
$$(\Z/p^{\max\{a,b\}}\Z)^*\hookrightarrow (\Oc_K/\p^\alpha\bar\p^\beta)^*$$
and
$$(\Oc_K/\p^\alpha\bar\p^\beta)^*\cong (\Oc_K/\p^\alpha)^* \times (\Oc_K/\bar\p^\beta)^*.$$
Notice that $1+p$ will map to $(1+p, 1+p)$ after composing the two maps above. Since we assume that the $p$-part of the conductor of $\psi_f$ is $\p^\alpha\bar\p^\beta$, we get that $\psi_f(1+p)$ must be a primitive $p^{\alpha-1}$-th root of unity, as well as a $p^{\beta-1}$-th one. This can of course only happen if $\alpha=\beta=n$.

Let's go back to counting all possible characters of $O_K^*$ of conductor $\p^n\bar\p^n$. Having in mind the isomorphism above, any such character is completely determined by the images of the generators of $ (\Oc_K/\p^\alpha)^*$ and $(\Oc_K/\bar\p^\beta)^*$. Wlog we can pick $1+p$ and $\delta$, of order $p^{n-1}$ and $p-1$ respectively, to generate both groups. Like before, the restriction of $\psi_f$ to $\Z$ is determined by the nebentypus. For the $p$-part we have that the two are in fact equal and since the nebentypus has trivial $p$-part we get the same for the $p$-part of $\psi_f|\Z$. This means that $(1+p, 1+p)$ and $(\delta, \delta)$ should map to $1$ and we therefore have that
$$\psi_f((1+p,1))=\psi_f((1,1+p))^{-1}$$
and
$$\psi_f((\delta,1))=\psi_f((1,\delta))^{-1}.$$
Apriori, $\psi_f((\delta,1))$ has $p-1$ choices. If however $n=1$, then this $\delta$ is the only generator and if it has trivial image, the conductor then becomes lower, which leaves $p-2$ choices. If $n>1$, the conductor condition is satisfied by restricting $\psi_f((1+p,1))$ to be a primitive $p^{n-1}$-th root of unity. This gives $p-1$ choices for the image of $\delta$ and $\varphi(p^{n-1})=p^{n-2}(p-1)$ many choices for that of $1+p$.

\subsubsection{p ramified} In this case $p=\p^2$ and $\Nm(\p)=p$. Assume that $p^t||d$ (remember that $N$ and $D_K$ are coprime). Assuming that $p^u||D_K$ (and since $p$ is ramified we also have that $u\geq1$), we get that $p^{t-u}||\Nm(\mathfrak{m})$ so the $p$-part of the conductor of $\psi_f$ should be $\p^{t-u}$. If $t=1$, we clearly have a unique choice for the $p$-part of $\psi_f$ which happens to match the Nebentypus too. If $t=2$, we are looking for non-trivial characters of $(\Oc_K/\p)^*$, which are apriori $p-2$ many. The nebentypus condition determines these characters uniquely on $(\Z/p\Z)^*$, which happens to be isomorphic to  $(\Oc_K/\p)^*$. The unique character that is left is actually non-trivial so we have a unique choice.

Summarizing all of the above:
$$CM(p^t)=\begin{cases}
1&t=0\\
1&t=1\textrm{ and $p$ ramified}\\
0&t=1\textrm{ and $p$ unramified}\\
1&t=2\textrm{ and $p$ ramified}\\
p-2&t=2\textrm{ and $p$ split}\\
p&t=2\textrm{ and $p$ inert}\\
0&t=2n+1\geq3\\
p^n(p-1)^2&t=2n\geq4\textrm{ and $p$ split}\\
p^n(p^2-1)&t=2n\geq4\textrm{ and $p$ inert}
\end{cases}$$
Recall that it is easy to determine whether a prime is split, inert or ramified in $K$ simply computing $\Bigg(\frac{D_K}{p}\Bigg)$. Putting everything together we get:
\begin{equation}\label{eq:CMEisdim}
\dim S_k^{d\textrm{-sc,new}}(\Gamma_0(Nd), \omega_d)^{\textrm{CM}_K} = \prod_{p^t||Nd}CM(p^t).
\end{equation}
Notice that the formula above works for any $N$ coprime to $D_K$, not just square-free. Specializing to square-free $N$ we get:
$$\dim S_k^{d\textrm{-sc,new}}(\Gamma_0(Nd), \omega_d)^{\textrm{CM}_K}=\begin{cases}
1 & N = 1 \textrm{ and } \textrm{rad}(D_K)|d,\\
0 & \textrm{otherwise}.
\end{cases}$$
where $\textrm{rad}(D_K)$ is the product of all primes dividing $D_K$.

Hopefully it should be obvious by now that we have sketched an algorithm that computes $\dim S_k^{\textrm{nG}}(\mathfrak{n})$ in an elementary way.

\subsection{The case $K=\Q(\sqrt{-p})$, $p\equiv3\mod 4$ and $\mathfrak{n}=(p)$.} \label{level a power of the discriminant}
Let $p\equiv3\mod 4$ be a prime number, and $(p) = \mathfrak{p}^2$ in $\Oc_K$
for $K=\Q(\sqrt{-p})$.
In this section, we essentially describe all the ingredients for a dimension formula in the cases where the level $\mathfrak{n}$ is $\mathfrak{p}$ or $\mathfrak{p}^2$.
For this purpose, we determine the parameters $I_i(\sigma_p)$ for all the suitable $\sigma_p$ in these two cases, 
further the twists of Base-Change, and show that there is no CM outside Base-Change. 
For simplicity, we restrict ourselves to the case where the class number $h_K$ of $K$ is 1.

Let us begin with $\mathfrak{n}=\mathfrak{p}$. In this case, the only possible type over $K$ of level $\mathfrak{p}$ is unramified Steinberg. 
It is fairly easy to see that the only type over $\Q$ that can base-change to this is unramified Steinberg at $p$. This space over $\Q$ coincides with the new $\Gamma_0(p)$ space and the $\sigma$-parameters can be computed using (\ref{eq:I1new}) - (\ref{eq:I5new}):
$$I_1 = p-1, I_2 = 0, I_3= \Big(\frac{-3}{p}\Big) -1, I_4= -2, I_5 = -1.$$
Notice that the new $\Gamma_0(p)$ space over $\Q$ contains no CM forms.
Using the same arguments as in Theorem \ref{thm:twist-2} we see  that there are no twists of Base-Change in $\Gamma_0(\mathfrak{p})$. We also claim that there are no CM forms: Indeed, assume there exists one of level $\Gamma_0(\mathfrak{p})$, $f$ say. 
Then $f$ is automorphic induction of $\psi$ from $L$ to $K$,
% f = AI_L^K(\psi)
where $L/K$ is a quadratic extension and $\psi$ is a Hecke character over~$L$. 
Since the L-series for $f$ and $\psi$ must be the same, we get that $L$ must be ramified at $\mathfrak{p}$ only. This is absurd since class field theory for $K$ tells us that it does not have any even degree abelian extensions ramified only at $\mathfrak{p}$. Summing everything up:
$$\dim S_k^{\textrm{nG}}(\mathfrak{p}) = \dim S_k^{\textrm{new}}(\Gamma_0(p)).$$
One can then use the $I_i$ parameters given above or use classical dimension formulas to compute the right hand side.

We move to the $\mathfrak{n}=\mathfrak{p}^2=(p)$ case. Let $\omega_\mathfrak{p}$ be the quadratic character of conductor $\mathfrak{p}$ and recall that $\omega_p$ is the quadratic character of conductor $p$. The following list summarizes the possible types over $K$ of level $\mathfrak{p}^2$ and for each one gives the possible types over $\Q$ 
that base-change to them:

\begin{itemize}
\item All the spaces listed in the trivial level case: After twisting their Base-Change by a quadratic character of conductor $\mathfrak{p}$ they become of level $\mathfrak{p}^2$.
The type at $\mathfrak{p}$ after Base-Change is again Principal Series, $I(\omega_\mathfrak{p}, \omega_\mathfrak{p})$.
\item Principal Series $I(\chi, \chi^{-1})$, where $\chi$ is of conductor $p$, non quadratic. Their Base-Change (or their twist by a suitable character of conductor $\mathfrak{p}$) becomes a newform of level $\mathfrak{p}^2$ and trivial nebentypus. The type at $\mathfrak{p}$
after Base-Change is again Principal Series, $I(\eta, \eta^{-1})$ with $\eta\neq \omega_\mathfrak{p}$.
\item Unramified Steinberg Series $\St(p)$. As mentioned in the level $\mathfrak{n}=\mathfrak{p}$ case, their Base-Change image are the Unramified Steinberg Series $\St(\mathfrak{p})$ over $K$. After twisting by a quadratic character of conductor $\mathfrak{p}$ they become newforms of level $\mathfrak{p}^2$ and trivial nebentypus. The type at $\mathfrak{p}$ after Base-Change is ramified Steinberg, $\omega_\mathfrak{p}\otimes\St(\mathfrak{p})$.
\item Finally all Supercuspidal series of level $p^2$ that are not part of the third case described in the trivial level situation. Their Base-Change image (and their twists by $\omega_\mathfrak{p}$) are Supercuspidal Series of level $\mathfrak{p}^2$ and trivial nebentypus. The type at $\mathfrak{p}$ after Base-Change is again Supercuspidal Series.
\end{itemize}

These four components comprise all the possible Base-Change as well as twists of it that can occur for level $\mathfrak{p}^2$ and trivial nebentypus. Table \ref{tbl:psquaredtypes} provides the formulas to compute the parameters $I_i(\sigma)$ in each case, which one needs in order to use Proposition \ref{prop:key} to compute the corresponding dimensions. Here $SC_3(p)$ and $SC_4(p)$ are the functions $\Tr S_3(p)$ and $\Tr S_4(p)$ respectively that are defined in the statement of Lemma 4.19 of \cite{FGT}, and $\textrm{CPS}(p)$ is defines as follows: 
$$\textrm{CPS}(p)=\begin{cases}
-2 &p\equiv 1 \mod 3\\
0 & \textrm{otherwise}\\
\end{cases}
$$

\begin{table}
\small
\begin{tabular}{|l||r|r|r|r|r|}
\hline
Type 											&dim $\sigma$		&dim $\sigma^{U_N}$	&tr$S_3$															&tr$S_4$								&dim $\sigma^{G_N}$ \\
\hline
$I(\omega_{\mathfrak{p}}, \omega_{\mathfrak{p}})$ 	&$\frac{p+1}{2}$		&$p-3$				&$1+\frac{1}{2}\textrm{SC}_3(p)$									&$0$								&$0$ \\
$I(\eta, \eta^{-1}), \eta \neq \omega_{\mathfrak{p}}$	&$\frac{(p-3)(p+1)}{2}$	&$1+h_K$			&$\textrm{CPS}(p)$												&$1+\frac{1}{2}\textrm{SC}_4(p)$		&$0$ \\
$\omega_{\mathfrak{p}}\otimes \St(\mathfrak{p})$		&$p-1$				&$0$				&$\Bigg(\frac{-3}{p}\Bigg)-1$										&$\Bigg(\frac{-1}{p}\Bigg)-1$			&$-1$ \\
Supercuspidal										&$\frac{(p-3)(p-1)}{2}$	&$p-2+h_K$			&$-2\Bigg(\frac{-3}{p}\Bigg)-\textrm{CPS}(p)-\textrm{SC}_3(p)$	&$-1-\frac{1}{2}\Bigg(\frac{-1}{p}\Bigg)$	&$0$ \\
\hline
\end{tabular}
\normalsize
\caption{Parameter values for the types contributing to Base-Change of level $\mathfrak{p}^2$}\label{tbl:psquaredtypes}
\end{table}

Finally we need to account for any CM forms not already present in the Base-Change subspace. In fact one can easily see that there are none. The argument is almost identical to the one given for the level $\mathfrak{p}$ case: The existence of any such CM form would imply the existence of a quadratic extension of $K$ ramified only at $\mathfrak{p}$ but no such extension exists.
At this point, anyone wanting to compute $\dim S_k^{\textrm{nG}}(\mathfrak{p}^2)$ has everything needed to do so.

\section{Final Remarks} \label{Final Remarks}
We would like to sum up here the cases for which we provide a complete answer:
\begin{itemize}
\item Odd Class Number, square-free level, coprime to the discriminant, trivial nebentypus.
\item Class Number $1$, $K=\Q(\sqrt{-p})$ for some prime $p\equiv 3\mod 4$ and level $\mathfrak{p}$ or $\mathfrak{p}^2$ where $\mathfrak{p}^2=(p)$.
\end{itemize}

Many of our arguments however provide partial answers to a broader range of cases:
\begin{itemize}
\item We have a complete description in (\ref{eq:I1new}) - (\ref{eq:I5new}) of the $\sigma$ parameters away from the discriminant as long as the nebentypus is trivial for any level, not just square-free ones. The generalization to non-trivial nebentypus should be quite straightforward but even more cumbersome to write down as single formulas.
\item For primes $p$ dividing the discriminant, the $0$ exponent case is the one treated in \cite{FGT} and we provide an answer for exponents $1$ and $2$ and trivial nebentypus in the case $p\equiv3\mod 4$.
\item Theorem \ref{thm:twist-2} applies to imaginary quadratic fields of even class number too, not only to the ones with an odd class number.
\item Finally the formula in (\ref{eq:CMEisdim}) allows one to compute the dimension of the classical newforms that base-change to Eisenstein forms for any $N$ coprime to the discriminant and trivial nebentypus, not just for the square-free levels.
\end{itemize}

\section{Computational results for the spaces of genuine forms} \label{sec:computational}
The authors have used the software \textit{Bianchi.gp} (\cite{Rahm11}, \cite{Rahm13})
to compute the necessary geometric-topological information about the whole Bianchi group,
and then applied a MAGMA \cite{magma} implementation by \c{S}eng\"un 
(for which we provide the algorithm in \cite{database})
to deduce the dimension of the relevant cohomology space for the congruence subgroup at level $\mathfrak{n}$,
passing by the Eckmann--Shapiro lemma.
They have then substracted the Eisenstein series space to get the cuspidal cohomology space,
which by the Eichler--Shimura(--Harder) isomorphism yields the cuspidal forms space.
Then they did substract the oldforms using a well-known recursive formula,
to get the dimensions of the newforms spaces $S_k(\mathfrak{n})$ defined in Section~\ref{sec:setting}.

\subsection{Level One}
We recall here a dimension computation of 4986 different spaces of cuspidal newforms at level $1$,
at varying discriminant $D$ (over 186 different imaginary quadratic fields) and varying weight $k+2$. 
The precise scope of our computations is given in~\cite{RahmSengun}. 
In only 22 of these spaces were we able to observe genuine forms. 
The precise data about these exceptional cases is provided in Table~\ref{exceptional}.
We note that in \cite{RahmSengun}, 
some further subspaces are tabulated, which are in fact populated by CM-forms (arising through automorphic induction).
We need this table in order to spot the twists of genuine level One forms to deeper levels.

\begin{table}
\begin{tabular}{|c|c|c|c|c|c|c|c|c|c|c|c|} \hline
$|D|$     &{\bf 7}   &{\bf 11} &{\bf 71} &{\bf 87}  &{\bf 91}   &{\bf 155}    &{\bf 199}   &{\bf 223}    &{\bf 231} &{\bf 339}  &{\bf 344}\\ \hline
$k+2$      &14         &12       &3        &4         &8          &6            &3           &2            &6         &3          &3          \\ \hline
$\dim$   &2          &2        &2        &2         &2          &2            &4           &2            &2         &2          &2         \\ \hline
\multicolumn{11}{c}{}\\ \hline

$|D|$     &{\bf 407}    &{\bf 415}     &{\bf 455}  &{\bf 483} &{\bf 571}  &{\bf 571}  &{\bf 643}   &{\bf 760} &{\bf 1003} &{\bf 1003 } &{\bf 1051}   \\ \hline
$k+2$       &2            &2             &2          &3         &2          &3          &2           &4         &2          &3           &2           \\ \hline
$\dim$    &2            &2             &2          &2         &2          &2          &2           &2         &2          &2           &2          \\ \hline
\end{tabular}
\caption{Level One cases where there are genuine classes}\label{exceptional}
\end{table}

\subsection{Level: the square of minus the discriminant}
We have compared the evaluation of our above formulas against our numerical results for the full space of cuspidal newforms.
At discriminant $-m$, at level $(m)$ (of norm $m^2$ and Hermite Normal Form (HNF) $[m^2,0,m]$)
and its divisors, we have carried out this computation from (automorphic form) weight~$2$ up to the following upper limits for the weight.
$$\begin{array}{|c|c|c|c|c|c|c|}
\hline
\text{Discriminant} & -7 & -11 & -19 & -43 & -67 \\
\hline
\text{weight up to} & 25 &  21 &  11 &  4  & 2  \\
\hline
\end{array}$$
The result is that of these 83 cuspidal newforms spaces, all are completely exhausted by (twists of) Base-Change, except for the following.
$$\begin{array}{|c|c|c|c|c|c|c|}
\hline
\text{Discriminant}                       & -7     & -7 & -11 & -11    & -11    & -43 \\
\hline
\text{weight}                             &  6     & 14 &  12 & 3      &  5     &  2 \\
\hline
\text{potentially genuine space dimension}& \bf{2} & 2  &  2  & \bf{4} & \bf{4} &  \bf{2} \\
\hline
\end{array}$$
Out of these, the forms at discriminant $-7$, weight $14$ and at discriminant $-11$, weight $12$ 
are twists of the genuine level One Bianchi modular forms already found by Grunewald~\cite{FGT}.
In the other cases, there are no level One forms that could be twisted~\cite{RahmSengun}.
That there are no CM-forms in the cases $m$ congruent to $3 \mod 4$, 
is guaranteed for arbitrary weight by Theorem~\ref{prop:formula}.
So the remaining spaces must be genuine, and we print them in boldface.

% \vfill

\subsection{Square-free levels} \label{sec:Square-free levels}

\begin{table}
 \centering
\begin{tabular}{|c|cc|cc|}
\hline  & Weight 2  &  & Higher weight (at least 3)&  \\
\hline Discriminant & \# spaces & \# genuine spaces & \# spaces & \# genuine spaces \\
\hline
-7  &  1174  &  355  &  556  &  17  \\
-11  &  1307  &  353  &  683  &  14  \\
-19  &  504  &  151  &  531  &  6  \\
-43  &  318  &  61  &  103  &  1  \\
-67  &  123  &  17  &  33  &  1 \\
-163  &  24  &  4   &  3  &  0  \\
\hline %&&&&\\
\end{tabular}
\caption{Overview of the sample presented in the Appendix.
Under ``\# spaces'', we count the newforms spaces that have been computed at the specified discriminant and varying level,
and under ``\# genuine spaces'', we count those of them which admit a non-trivial genuine subspace.}
\label{overview}
\end{table}

Let $n \in \Z$ be square-free and coprime to the discriminant of the imaginary quadratic field in question.
Consider the level $(n)$ of HNF $[n^2, 0, n]$. 
Then we have the formulas of Section~\ref{Dimension formulas} for the dimension of the space of base-changed forms of level $(n)$.
We have compared them against the machine computed dimension of the space of cuspidal newforms.
{At discriminant -19,} the range of this machine computation was as follows.
$$\begin{array}{|c|c|c|c|c|c|c|c|c|c|c|c|c|}
\hline
\text{Levels }n \text{ at discriminant -19}& 2,3 & 6  & 11 & 13 & 5  & 7, 15  & 14, 17 & 23 & 10 & 30 & 22, 31, 33\\
\hline
\text{weight up to}                        & 22   & 21 & 15 & 13 & 12 & 11     & 10     &  8 &  5 &  4 &  2\\
\hline
\end{array}$$
Out of these 154 spaces, only the following six spaces can admit genuine forms:
$$\begin{array}{|c|c|c|c|c|c|c|}
\hline
\text{Level}                              & 6    & 6 & 11 & 15     & 17     & 30 \\
\hline
\text{weight}                             & 3    & 4 & 2  &  2     & 2      & 2  \\
\hline
\text{genuine space dimension}            & 2    & 2 &  2  & 2     & 2      & 4 \\
\hline
\end{array}$$
As the level is square-free and the class group of the imaginary quadratic field is trivial, there are no twists of base-change forms.
By Theorem~\ref{thm:twist-2}, there are no CM forms.
So the above six spaces must be constituted of genuine forms.

The above sample of 154 spaces allows us to guess that genuine forms are more likely to occur at low weights than at high weights (supported by Table~\ref{overview});
and that levels which admit genuine forms at some low weights are rather unlikely to admit more of them at higher weights 
(supported by the Appendix).

% \newpage

\section{Appendix: Detailed results in the square-free level case}

In addition to the sample of Section~\ref{sec:Square-free levels},
we include in the following tables a range of square-free ideals which are not Galois-stable (in our setting of imaginary quadratic fields, not totally real).
At each discriminant, we first specify the range of the pertinent machine computation, 
and then in a separate table the spaces with a non-trivial genuine subspace.
The computation usually was run also at the Galois-conjugated level, but in the range tables, 
we print only one HNF per pair of Galois-conjugate levels.
We omit the details for weight $2$, and outsource them to~\cite{database}.
A statistical overview is given in Table~\ref{overview}.
The bottleneck for the computations were the memory requirements; processor time aspects were completely eclipsed by them. 
With the quadratic (in the weight) growth of the coefficient modules, the MAGMA program did grow its memory requirements quadratically.

\begin{center}
\footnotesize \begin{tabular}{|c|c|} \hline {\bf Range at discriminant  -19.} Level HNFs, up to Galois conjugacy & {weights} \\
\hline 
164  levels of norm up to  1145
 and their Galois conjugates & only 2 \\ 
\hline \begin{tabular}{c} $ 
[ 140, 10, 2 ]
,
[ 140, 30, 2 ]
,
[ 153, 39, 3 ]
,
[ 161, 12, 1 ]
,
[ 161, 127, 1 ]
,
[ 172, 28, 2 ]
,
[ 175, 25, 5 ]
,
[ 187, 13, 1 ]
,
$ \\ 
 $
[ 187, 156, 1 ]
,
[ 188, 12, 2 ]
,
[ 191, 154, 1 ]
,
[ 197, 150, 1 ]
,
[ 199, 128, 1 ]
,
[ 207, 30, 3 ]
,
[ 215, 100, 1 ]
,
[ 215, 14, 1 ]
,
$ \\ 
 $
[ 220, 38, 2 ]
,
[ 220, 48, 2 ]
,
[ 229, 139, 1 ]
,
[ 233, 166, 1 ]
,
[ 235, 100, 1 ]
,
[ 235, 194, 1 ]
,
[ 239, 204, 1 ]
,
[ 244, 106, 2 ]
,
$ \\ 
 $
[ 245, 0, 7 ]
,
[ 251, 198, 1 ]
,
[ 252, 30, 6 ]
,
[ 253, 173, 1 ]
,
[ 253, 217, 1 ]
,
[ 263, 113, 1 ]
,
[ 271, 227, 1 ]
,
[ 275, 10, 5 ]
,
$ \\ 
 $
[ 277, 16, 1 ]
,
[ 283, 183, 1 ]
,
[ 289, 207, 1 ]
,
[ 292, 100, 2 ]
,
[ 301, 229, 1 ]
,
[ 301, 243, 1 ]
,
[ 305, 114, 1 ]
,
[ 305, 129, 1 ]
,
$ \\ 
 $
[ 311, 17, 1 ]
,
[ 313, 273, 1 ]
,
[ 347, 18, 1 ]
,
[ 349, 110, 1 ]
,
[ 353, 131, 1 ]
,
[ 359, 140, 1 ]
,
[ 367, 303, 1 ]
,
[ 389, 168, 1 ]
,
$ \\ 
 $
$ \end{tabular} & up to  3  \\ 
\hline \begin{tabular}{c} $ 
[ 101, 26, 1 ]
,
[ 115, 10, 1 ]
,
[ 115, 35, 1 ]
,
[ 119, 47, 1 ]
,
[ 119, 54, 1 ]
,
[ 121, 24, 1 ]
,
[ 131, 105, 1 ]
,
[ 137, 11, 1 ]
,
$ \\ 
 $
[ 139, 56, 1 ]
,
[ 149, 58, 1 ]
,
[ 157, 115, 1 ]
,
[ 163, 75, 1 ]
,
[ 68, 26, 2 ]
,
[ 77, 19, 1 ]
,
[ 77, 68, 1 ]
,
[ 83, 37, 1 ]
,
$ \\ 
 $
[ 85, 20, 1 ]
,
[ 85, 30, 1 ]
,
[ 900, 0, 30 ]
,
[ 92, 20, 2 ]
,
[ 99, 24, 3 ]
,
$ \end{tabular} & up to  4  \\ 
\hline \begin{tabular}{c} $ 
[ 100, 0, 10 ]
,
[ 180, 0, 6 ]
,
[ 35, 15, 1 ]
,
[ 35, 29, 1 ]
,
[ 43, 14, 1 ]
,
[ 44, 16, 2 ]
,
[ 47, 40, 1 ]
,
[ 49, 15, 1 ]
,
$ \\ 
 $
[ 55, 19, 1 ]
,
[ 55, 24, 1 ]
,
[ 61, 53, 1 ]
,
[ 63, 15, 3 ]
,
[ 73, 22, 1 ]
,
$ \end{tabular} & up to  5  \\ 
\hline \begin{tabular}{c} $ 
[ 20, 0, 2 ]
,
[ 25, 20, 1 ]
,
$ \end{tabular} & up to  6  \\ 
\hline \begin{tabular}{c} $ 
[ 23, 10, 1 ]
,
[ 529, 0, 23 ]
,
$ \end{tabular} & up to  8  \\ 
\hline \begin{tabular}{c} $ 
[ 196, 0, 14 ]
,
[ 289, 0, 17 ]
,
$ \end{tabular} & up to  10  \\ 
\hline \begin{tabular}{c} $ 
[ 17, 13, 1 ]
,
[ 225, 0, 15 ]
,
[ 28, 10, 2 ]
,
[ 49, 0, 7 ]
,
[ 7, 1, 1 ]
,
$ \end{tabular} & up to  11  \\ 
\hline \begin{tabular}{c} $ 
[ 25, 0, 5 ]
,
[ 45, 0, 3 ]
,
[ 5, 0, 1 ]
,
$ \end{tabular} & up to  12  \\ 
\hline \begin{tabular}{c} $ 
[ 169, 0, 13 ]
,
$ \end{tabular} & up to  13  \\ 
\hline \begin{tabular}{c} $ 
[ 121, 0, 11 ]
,
$ \end{tabular} & up to  15  \\ 
\hline \begin{tabular}{c} $ 
[ 11, 2, 1 ]
,
$ \end{tabular} & up to  16  \\ 
\hline \begin{tabular}{c} $ 
[ 36, 0, 6 ]
,
$ \end{tabular} & up to  21  \\ 
\hline \begin{tabular}{c} $ 
[ 4, 0, 2 ]
,
[ 9, 0, 3 ].
$ \end{tabular} & up to  22  \\ 
\hline \end{tabular} \normalsize
\end{center}

Out of these, we get the following genuine spaces:
\begin{center}
\footnotesize \begin{tabular}{|c|c|c|} \hline {Weight} & $d$ & Level HNFs at discriminant  -19  with 
genuine space of dim. $d$ \\
\hline  2  &  1  & 
73  levels of norm up to  1145
 \\ 
\hline  2  &  2  & 
31  levels of norm up to  1099
 \\ 
\hline  2  &  3  & 
16  levels of norm up to  935
 \\ 
\hline  2  &  4  & 
17  levels of norm up to  1081
 \\ 
\hline  2  &  5  & 
6  levels of norm up to  932
 \\ 
\hline  2  &  6  & 
6  levels of norm up to  940
 \\ 
\hline  2  &  7  & 
2  levels of norm up to  955
 \\ 
\hline  3  &  2  & 
\begin{tabular}{c} $ 
[ 289, 207, 1 ]
,
[ 289, 81, 1 ]
,
[ 36, 0, 6 ]
,
[ 49, 15, 1 ]
,
[ 49, 33, 1 ]
,
$ \end{tabular}  \\ 
\hline  4  &  2  & 
\begin{tabular}{c} $ 
[ 36, 0, 6 ]
.
$ \end{tabular}  \\ 
\hline \end{tabular} \normalsize
\end{center}
That is, the genuine spaces at level $(6) = [36, 0, 6]$ already observed in Section~\ref{sec:Square-free levels},
as well as two further two-dimensional weight $3$ spaces and their Galois conjugates.

% \newpage

\begin{center}
\footnotesize \begin{tabular}{|c|c|} \hline {\bf Range at discriminant -43.} Level HNFs, up to Galois conjugacy & weights \\
\hline 
133  levels of norm up to  787
 and their Galois conjugates & only 2 \\ 

\hline \begin{tabular}{c} $ 
[ 52, 22, 2 ]
,
[ 59, 31, 1 ]
,
[ 67, 7, 1 ]
,
[ 79, 42, 1 ]
,
[ 68, 28, 2 ]
,
[ 83, 74, 1 ]
,
[ 103, 69, 1 ]
,
[ 109, 71, 1 ]
,
$ \\ 
 $
[ 97, 64, 1 ]
,
[ 107, 49, 1 ]
,
[ 121, 110, 1 ]
,
[ 100, 0, 10 ]
,
[ 101, 9, 1 ]
,
[ 117, 33, 3 ]
,
[ 92, 38, 2 ]
,
[ 127, 40, 1 ]
,
$ \\ 
 $
[ 99, 30, 3 ]
,
[ 121, 0, 11 ]
,
[ 139, 48, 1 ]
,
[ 167, 154, 1 ]
,
$ \end{tabular} & up to  3  \\ 
\hline \begin{tabular}{c} $ 
[ 53, 46, 1 ]
,
[ 31, 4, 1 ]
,
[ 41, 5, 1 ]
,
[ 49, 0, 7 ]
,
[ 44, 0, 2 ]
,
[ 47, 22, 1 ]
,
$ \end{tabular} & up to  4  \\ 
\hline \begin{tabular}{c} $ 
[ 17, 2, 1 ]
,
[ 23, 3, 1 ]
,
[ 25, 0, 5 ]
,
$ \end{tabular} & up to  5  \\ 
\hline \begin{tabular}{c} $ 
[ 13, 11, 1 ]
,
[ 11, 0, 1 ]
,
$ \end{tabular} & up to  6  \\ 
\hline \begin{tabular}{c} $ 
[ 9, 0, 3 ]
,
$ \end{tabular} & up to  7  \\ 
\hline \begin{tabular}{c} $ 
[ 4, 0, 2 ]
.
$ \end{tabular} & up to  8  \\ 
\hline \end{tabular} \normalsize
\end{center}
Out of these, we get the following genuine spaces:
\begin{center}
\footnotesize \begin{tabular}{|c|c|c|} \hline
{Weight} & $d$ & Level HNFs at discriminant  -43  with 
genuine space of dim. $d$ \\
\hline  2  &  1  & 
32  levels of norm up to  737
 \\ 
\hline  2  &  2  & 
18  levels of norm up to  713
 \\ 
\hline  2  &  3  & 
4  levels of norm up to  719
 \\ 
\hline  2  &  4  & 
7  levels of norm up to  572
 \\ 
\hline  6  &  2  & 
\begin{tabular}{c} $ 
[ 9, 0, 3 ]
.
$ \end{tabular}  \\ 
\hline \end{tabular} \normalsize
\end{center}

% \vfill

\begin{center}
\footnotesize \begin{tabular}{|c|c|} \hline {\bf Range at discriminant -67}. Level HNFs, up to Galois conjugacy & weights \\
\hline 
54  levels of norm up to  361
 and their Galois conjugates & only 2 \\ 

\hline \begin{tabular}{c} $ 
[ 29, 3, 1 ]
,
[ 37, 4, 1 ]
,
[ 36, 0, 6 ]
,
[ 49, 0, 7 ]
,
[ 47, 41, 1 ]
,
[ 59, 6, 1 ]
,
[ 71, 36, 1 ]
,
$ \end{tabular} & up to  3  \\ 
\hline \begin{tabular}{c} $ 
[ 17, 0, 1 ]
,
[ 19, 1, 1 ]
,
[ 23, 2, 1 ]
,
[ 25, 0, 5 ]
,
$ \end{tabular} & up to  4  \\ 
\hline \begin{tabular}{c} $ 
[ 9, 0, 3 ]
,
$ \end{tabular} & up to  5  \\ 
\hline \begin{tabular}{c} $ 
[ 4, 0, 2 ]
.
$ \end{tabular} & up to  6  \\ 
\hline \end{tabular} \normalsize
\end{center}
Out of these, we get the following genuine spaces:
\begin{center}
\footnotesize \begin{tabular}{|c|c|c|} \hline 
{Weight} & $d$ & Level HNFs at discriminant  -67  with 
genuine space of dim. $d$ \\
\hline  2  &  1  & 
6  levels of norm up to  323
 \\ 
\hline  2  &  2  & 
3  levels of norm up to  289
 \\ 
\hline  2  &  3  & 
6  levels of norm up to  289
 \\ 
\hline  2  &  4  & 
one level of norm 121
 \\ 
\hline  2  &  8  & 
one level of norm  196
 \\ 
\hline  3  &  2  & 
\begin{tabular}{c} $ 
[ 36, 0, 6 ]
.
$ \end{tabular}  \\ 
\hline \end{tabular} \normalsize
\end{center}

% \vfill

\begin{center}
\footnotesize \begin{tabular}{|c|c|} \hline {\bf Range at discriminant  -163.} 
Level HNFs, up to Galois conjugacy & {weights} \\
\hline \begin{tabular}{c} $ 
[ 25, 0, 5 ]
,
[ 36, 0, 6 ]
,
[ 41, 40, 1 ]
,
[ 43, 41, 1 ]
,
[ 49, 0, 7 ]
,
[ 47, 2, 1 ]
,
[ 53, 49, 1 ]
,
[ 61, 4, 1 ]
,
$ \\ 
 $
[ 71, 5, 1 ]
,
[ 83, 6, 1 ]
,
[ 97, 7, 1 ]
,
[ 121, 0, 11 ]
,
[ 113, 8, 1 ]
,
$ \end{tabular} & 2 \\ 
\hline \begin{tabular}{c} $ 
[ 9, 0, 3 ]
,
$ \end{tabular} & up to  3  \\ 
\hline \begin{tabular}{c} $ 
[ 4, 0, 2 ]
.
$ \end{tabular} & up to  4  \\ 
\hline \end{tabular} \normalsize
\end{center}
Out of these, we get the following genuine spaces:
\begin{center}
 \footnotesize \begin{tabular}{|c|c|c|} \hline 
{Weight} & $d$ & Level HNFs at discriminant  -163  with 
genuine space of dim. $d$ \\
\hline  2  &  2  & 
3  levels of norm up to  47
 \\ 
\hline  2  &  4  & 
[49, 0, 7]
 \\ 
\hline \end{tabular} \normalsize
\end{center}

% \vfill

\begin{center}
 \scriptsize \begin{tabular}{|c|c|} \hline {\bf Range at discriminant  -7.} Level HNFs, up to Galois conjugacy & {weights} \\
\hline 
460  levels of norm up to  2767
 and their Galois conjugates & only 2 \\ 
\hline \begin{tabular}{c} $ 
[ 214, 155, 1 ]
,
[ 218, 188, 1 ]
,
[ 218, 79, 1 ]
,
[ 214, 165, 1 ]
,
[ 212, 28, 2 ]
,
[ 172, 36, 2 ]
,
[ 253, 59, 1 ]
,
[ 198, 18, 3 ]
,
$ \\ 
 $
[ 253, 147, 1 ]
,
[ 254, 104, 1 ]
,
[ 254, 231, 1 ]
,
[ 198, 12, 3 ]
,
[ 226, 42, 1 ]
,
[ 226, 70, 1 ]
,
[ 242, 0, 11 ]
,
[ 275, 20, 5 ]
,
$ \\ 
 $
[ 298, 34, 1 ]
,
[ 261, 21, 3 ]
,
[ 274, 16, 1 ]
,
[ 274, 153, 1 ]
,
[ 277, 253, 1 ]
,
[ 281, 33, 1 ]
,
[ 268, 22, 2 ]
,
[ 289, 0, 17 ]
,
$ \\ 
 $
[ 298, 183, 1 ]
,
[ 284, 78, 2 ]
,
[ 326, 25, 1 ]
,
[ 394, 341, 1 ]
,
[ 347, 272, 1 ]
,
[ 317, 233, 1 ]
,
[ 319, 94, 1 ]
,
[ 319, 268, 1 ]
,
$ \\ 
 $
[ 302, 220, 1 ]
,
[ 302, 232, 1 ]
,
[ 326, 137, 1 ]
,
[ 359, 128, 1 ]
,
[ 361, 0, 19 ]
,
[ 373, 154, 1 ]
,
[ 379, 27, 1 ]
,
[ 148, 16, 2 ]
,
$ \\ 
 $
[ 358, 125, 1 ]
,
[ 358, 304, 1 ]
,
[ 401, 248, 1 ]
,
[ 331, 174, 1 ]
,
[ 382, 19, 1 ]
,
[ 407, 378, 1 ]
,
[ 389, 296, 1 ]
,
[ 407, 193, 1 ]
,
$ \\ 
 $
[ 387, 72, 3 ]
,
[ 382, 210, 1 ]
,
[ 333, 24, 3 ]
,
[ 431, 389, 1 ]
,
[ 386, 73, 1 ]
,
[ 386, 119, 1 ]
,
[ 394, 144, 1 ]
,
[ 422, 401, 1 ]
,
$ \\ 
 $
[ 337, 212, 1 ]
,
[ 421, 244, 1 ]
,
[ 338, 0, 13 ]
,
[ 422, 190, 1 ]
,
[ 457, 85, 1 ]
,
[ 443, 285, 1 ]
,
[ 449, 196, 1 ]
,
[ 487, 103, 1 ]
,
$ \\ 
 $
[ 541, 46, 1 ]
,
[ 547, 459, 1 ]
,
[ 463, 80, 1 ]
,
[ 473, 61, 1 ]
,
[ 473, 147, 1 ]
,
[ 477, 114, 3 ]
,
[ 529, 0, 23 ]
,
[ 529, 32, 1 ]
,
$ \\ 
 $
[ 575, 45, 5 ]
,
[ 571, 158, 1 ]
,
[ 491, 234, 1 ]
,
[ 569, 178, 1 ]
,
[ 499, 151, 1 ]
,
[ 557, 133, 1 ]
,
[ 613, 563, 1 ]
,
[ 599, 129, 1 ]
,
$ \\ 
 $
$ \end{tabular} & up to  3  \\ 
\hline \begin{tabular}{c} $ 
[ 86, 18, 1 ]
,
[ 106, 38, 1 ]
,
[ 106, 14, 1 ]
,
[ 116, 14, 2 ]
,
[ 86, 61, 1 ]
,
[ 92, 26, 2 ]
,
[ 100, 0, 10 ]
,
[ 211, 190, 1 ]
,
$ \\ 
 $
[ 207, 27, 3 ]
,
[ 163, 25, 1 ]
,
[ 197, 144, 1 ]
,
[ 191, 19, 1 ]
,
[ 193, 119, 1 ]
,
[ 239, 72, 1 ]
,
[ 169, 0, 13 ]
,
[ 179, 53, 1 ]
,
$ \\ 
 $
[ 225, 0, 15 ]
,
[ 233, 30, 1 ]
,
[ 263, 123, 1 ]
,
[ 121, 105, 1 ]
,
[ 121, 0, 11 ]
,
[ 134, 78, 1 ]
,
[ 134, 122, 1 ]
,
[ 127, 104, 1 ]
,
$ \\ 
 $
[ 137, 16, 1 ]
,
[ 142, 39, 1 ]
,
[ 142, 31, 1 ]
,
[ 151, 81, 1 ]
,
[ 149, 114, 1 ]
,
[ 158, 91, 1 ]
,
[ 158, 145, 1 ]
,
$ \end{tabular} & up to  4  \\ 
\hline \begin{tabular}{c} $ 
[ 36, 0, 6 ]
,
[ 67, 55, 1 ]
,
[ 71, 31, 1 ]
,
[ 74, 45, 1 ]
,
[ 74, 65, 1 ]
,
[ 79, 66, 1 ]
,
[ 46, 32, 1 ]
,
[ 46, 36, 1 ]
,
$ \\ 
 $
[ 44, 8, 2 ]
,
[ 50, 5, 5 ]
,
[ 58, 36, 1 ]
,
[ 58, 7, 1 ]
,
[ 107, 48, 1 ]
,
[ 109, 29, 1 ]
,
[ 99, 12, 3 ]
,
[ 113, 70, 1 ]
,
$ \\ 
 $
$ \end{tabular} & up to  5  \\ 
\hline \begin{tabular}{c} $ 
[ 43, 18, 1 ]
,
[ 53, 14, 1 ]
,
$ \end{tabular} & up to  6  \\ 
\hline \begin{tabular}{c} $ 
[ 29, 7, 1 ]
,
[ 22, 17, 1 ]
,
[ 22, 15, 1 ]
,
[ 18, 3, 3 ]
,
[ 37, 28, 1 ]
,
$ \end{tabular} & up to  7  \\ 
\hline \begin{tabular}{c} $ 
[ 25, 0, 5 ]
,
[ 23, 13, 1 ]
,
$ \end{tabular} & up to  8  \\ 
\hline \begin{tabular}{c} $ 
[ 11, 4, 1 ]
,
$ \end{tabular} & up to  11  \\ 
\hline \begin{tabular}{c} $ 
[ 4, 0, 2 ]
,
$ \end{tabular} & up to  12  \\ 
\hline \begin{tabular}{c} $ 
[ 4, 2, 1 ]
,
$ \end{tabular} & up to  15  \\ 
\hline \begin{tabular}{c} $ 
[ 2, 0, 1 ]
,
$ \end{tabular} & up to  19  \\ 
\hline \begin{tabular}{c} $ 
[ 9, 0, 3 ]
.
$ \end{tabular} & up to  20  \\ 
\hline \end{tabular} \normalsize
\end{center}
Out of these, we get the following genuine spaces:
\begin{center}
\footnotesize \begin{tabular}{|c|c|c|} \hline {Weight} & $d$ & Level HNFs at discriminant  -7  with genuine space of dim. $d$ \\
\hline  2  &  1  &
200  levels of norm up to  2657
 \\
\hline  2  &  2  &
100  levels of norm up to  1913
 \\
\hline  2  &  3  &
30  levels of norm up to  1814
 \\
\hline  2  &  4  &
15  levels of norm up to  2563
 \\
\hline  2  &  5  &
6  levels of norm up to  1439
 \\
\hline  2  &  6  &
4  levels of norm up to  1702
 \\
\hline  3  &  2  &
\begin{tabular}{c} $
[ 225, 0, 15 ]
,
$ \end{tabular}  \\
\hline  4  &  1  &
\begin{tabular}{c} $
[ 11, 6, 1 ]
,
[ 11, 4, 1 ]
,
[ 22, 17, 1 ]
,
[ 22, 4, 1 ]
,
[ 46, 36, 1 ]
,
[ 46, 9, 1 ]
,
[ 92, 26, 2 ]
,
$ \\
 $
[ 92, 18, 2 ]
,
[ 121, 105, 1 ]
,
[ 121, 15, 1 ]
,
[ 116, 14, 2 ]
,
[ 116, 42, 2 ]
,
$ \end{tabular}  \\
\hline  4  &  2  &
\begin{tabular}{c} $
[ 22, 15, 1 ]
,
[ 22, 6, 1 ]
,
[ 58, 50, 1 ]
,
[ 58, 7, 1 ]
.
$ \end{tabular}  \\
% \hline  14  &  2  &
% \begin{tabular}{c} $
% [ 1, 0, 1 ]
% ,
% $ \end{tabular}  \\
\hline \end{tabular} \normalsize
\end{center}

% \vfill

\begin{center}
\footnotesize \begin{tabular}{|c|c|} \hline {\bf Range at discriminant  -11.} 
Level HNFs, up to Galois conjugacy & {weights} \\
\hline 
507  levels of norm up to  2803
 and their Galois conjugates & only 2 \\ 

\hline \begin{tabular}{c} $ 
[ 180, 18, 6 ]
,
[ 207, 12, 3 ]
,
[ 213, 14, 1 ]
,
[ 213, 156, 1 ]
,
[ 225, 0, 15 ]
,
[ 235, 161, 1 ]
,
[ 235, 208, 1 ]
,
[ 236, 102, 2 ]
,
$ \\ 
 $
[ 245, 21, 7 ]
,
[ 265, 118, 1 ]
,
[ 265, 171, 1 ]
,
[ 267, 125, 1 ]
,
[ 267, 230, 1 ]
,
[ 268, 48, 2 ]
,
[ 276, 100, 2 ]
,
[ 276, 54, 2 ]
,
$ \\ 
 $
[ 279, 27, 3 ]
,
[ 284, 112, 2 ]
,
[ 289, 0, 17 ]
,
[ 291, 126, 1 ]
,
[ 291, 261, 1 ]
,
[ 295, 228, 1 ]
,
[ 295, 243, 1 ]
,
[ 309, 120, 1 ]
,
$ \\ 
 $
[ 309, 17, 1 ]
,
[ 311, 280, 1 ]
,
[ 313, 244, 1 ]
,
[ 317, 224, 1 ]
,
[ 331, 104, 1 ]
,
[ 333, 39, 3 ]
,
[ 335, 158, 1 ]
,
[ 335, 243, 1 ]
,
$ \\ 
 $
[ 339, 102, 1 ]
,
[ 339, 123, 1 ]
,
[ 345, 156, 1 ]
,
[ 345, 18, 1 ]
,
[ 345, 248, 1 ]
,
[ 345, 303, 1 ]
,
[ 353, 320, 1 ]
,
[ 355, 156, 1 ]
,
$ \\ 
 $
[ 355, 298, 1 ]
,
[ 356, 104, 2 ]
,
[ 361, 0, 19 ]
,
[ 367, 128, 1 ]
,
[ 372, 142, 2 ]
,
[ 372, 166, 2 ]
,
[ 379, 335, 1 ]
,
[ 383, 19, 1 ]
,
$ \\ 
 $
[ 388, 134, 2 ]
,
[ 389, 177, 1 ]
,
[ 397, 165, 1 ]
,
[ 401, 148, 1 ]
,
[ 411, 195, 1 ]
,
[ 411, 332, 1 ]
,
[ 412, 170, 2 ]
,
[ 419, 124, 1 ]
,
$ \\ 
 $
[ 421, 35, 1 ]
,
[ 433, 386, 1 ]
,
[ 443, 361, 1 ]
,
[ 445, 141, 1 ]
,
[ 445, 36, 1 ]
,
[ 449, 183, 1 ]
,
[ 452, 20, 2 ]
,
[ 463, 207, 1 ]
,
$ \\ 
 $
[ 467, 179, 1 ]
,
[ 471, 422, 1 ]
,
[ 485, 126, 1 ]
,
[ 485, 223, 1 ]
,
[ 487, 169, 1 ]
,
[ 489, 134, 1 ]
,
[ 489, 191, 1 ]
,
[ 499, 412, 1 ]
,
$ \\ 
 $
[ 509, 22, 1 ]
,
[ 515, 188, 1 ]
,
[ 515, 223, 1 ]
,
[ 521, 39, 1 ]
,
[ 529, 0, 23 ]
,
[ 529, 119, 1 ]
,
[ 577, 184, 1 ]
,
[ 587, 281, 1 ]
,
$ \\ 
 $
[ 599, 181, 1 ]
,
[ 619, 536, 1 ]
,
[ 631, 168, 1 ]
,
[ 643, 283, 1 ]
,
$ \end{tabular} & up to  3  \\ 
\hline \begin{tabular}{c} $ 
[ 100, 0, 10 ]
,
[ 111, 23, 1 ]
,
[ 111, 50, 1 ]
,
[ 115, 18, 1 ]
,
[ 115, 41, 1 ]
,
[ 124, 18, 2 ]
,
[ 137, 58, 1 ]
,
[ 141, 114, 1 ]
,
$ \\ 
 $
[ 141, 120, 1 ]
,
[ 147, 0, 7 ]
,
[ 148, 26, 2 ]
,
[ 155, 133, 1 ]
,
[ 155, 71, 1 ]
,
[ 157, 108, 1 ]
,
[ 159, 12, 1 ]
,
[ 159, 65, 1 ]
,
$ \\ 
 $
[ 163, 134, 1 ]
,
[ 177, 110, 1 ]
,
[ 177, 125, 1 ]
,
[ 179, 109, 1 ]
,
[ 181, 116, 1 ]
,
[ 185, 13, 1 ]
,
[ 185, 161, 1 ]
,
[ 188, 40, 2 ]
,
$ \\ 
 $
[ 191, 137, 1 ]
,
[ 196, 0, 14 ]
,
[ 199, 167, 1 ]
,
[ 201, 158, 1 ]
,
[ 201, 176, 1 ]
,
[ 212, 24, 2 ]
,
[ 223, 173, 1 ]
,
[ 229, 101, 1 ]
,
$ \\ 
 $
[ 251, 203, 1 ]
,
[ 257, 159, 1 ]
,
[ 269, 205, 1 ]
,
[ 75, 0, 5 ]
,
[ 93, 21, 1 ]
,
[ 93, 83, 1 ]
,
$ \end{tabular} & up to  4  \\ 
\hline \begin{tabular}{c} $ 
[ 15, 8, 1 ]
,
[ 169, 0, 13 ]
,
[ 507, 0, 13 ]
,
[ 1521, 0, 39 ]
,
[ 103, 17, 1 ]
,
[ 113, 10, 1 ]
,
[ 36, 0, 6 ]
,
[ 45, 3, 3 ]
,
$ \\ 
 $
[ 60, 12, 2 ]
,
[ 60, 22, 2 ]
,
[ 67, 24, 1 ]
,
[ 69, 18, 1 ]
,
[ 69, 27, 1 ]
,
[ 71, 14, 1 ]
,
[ 89, 36, 1 ]
,
[ 92, 36, 2 ]
,
$ \\ 
 $
[ 97, 29, 1 ]
,
$ \end{tabular} & up to  5  \\ 
\hline \begin{tabular}{c} $ 
[ 47, 20, 1 ]
,
[ 49, 0, 7 ]
,
[ 53, 12, 1 ]
,
[ 59, 51, 1 ]
,
$ \end{tabular} & up to  6  \\ 
\hline \begin{tabular}{c} $ 
[ 2209, 0, 47 ]
,
[ 20, 2, 2 ]
,
[ 31, 21, 1 ]
,
[ 37, 13, 1 ]
,
$ \end{tabular} & up to  7  \\ 
\hline \begin{tabular}{c} $ 
[ 15, 11, 1 ]
,
[ 23, 18, 1 ]
,
$ \end{tabular} & up to  8  \\ 
\hline \begin{tabular}{c} $ 
[ 12, 0, 2 ]
,
$ \end{tabular} & up to  9  \\ 
\hline \begin{tabular}{c} $ 
[ 25, 0, 5 ]
,
$ \end{tabular} & up to  10  \\ 
\hline \begin{tabular}{c} $ 
[ 9, 2, 1 ]
,
$ \end{tabular} & up to  11  \\ 
\hline \begin{tabular}{c} $ 
[ 9, 0, 3 ]
,
$ \end{tabular} & up to  12  \\ 
\hline \begin{tabular}{c} $ 
[ 4, 0, 2 ]
,
$ \end{tabular} & up to  15  \\ 
\hline \begin{tabular}{c} $ 
[ 3, 0, 1 ]
,
$ \end{tabular} & up to  17  \\ 
\hline \begin{tabular}{c} $ 
[ 25, 16, 1 ]
,
[ 5, 1, 1 ]
.
$ \end{tabular} & up to  20  \\ 
\hline \end{tabular} \normalsize
\end{center}
Out of these, we get the following genuine spaces:
\begin{center}
\footnotesize \begin{tabular}{|c|c|c|} \hline 
{Weight} & $d$ & Level HNFs at discriminant  -11  with 
genuine space of dim. $d$ \\
\hline  2  &  1  & 
178  levels of norm up to  2491
 \\ 
\hline  2  &  2  & 
107  levels of norm up to  2621
 \\ 
\hline  2  &  3  & 
32  levels of norm up to  2689
 \\ 
\hline  2  &  4  & 
18  levels of norm up to  2201
 \\ 
\hline  2  &  5  & 
4  levels of norm up to  1335
 \\ 
\hline  2  &  6  & 
6  levels of norm up to  2597
 \\ 
\hline  2  &  7  & 
4  levels of norm up to  1035
 \\ 
\hline  2  &  12  & 
4  levels of norm up to  2209
 \\ 
\hline  4  &  1  & 
\begin{tabular}{c} $ 
[ 15, 6, 1 ]
,
[ 15, 8, 1 ]
,
[ 185, 161, 1 ]
,
[ 185, 23, 1 ]
,
[ 20, 2, 2 ]
,
[ 20, 6, 2 ]
,
[ 45, 3, 3 ]
,
$ \\ 
 $
[ 45, 9, 3 ]
,
$ \end{tabular}  \\ 
\hline  4  &  2  & 
\begin{tabular}{c} $ 
[ 100, 0, 10 ]
,
[ 92, 36, 2 ]
,
[ 92, 8, 2 ]
,
$ \end{tabular}  \\ 
\hline  4  &  5  & 
\begin{tabular}{c} $ 
[ 60, 12, 2 ]
,
[ 60, 16, 2 ]
,
$ \end{tabular}  \\ 
\hline  6  &  2  & 
\begin{tabular}{c} $ 
[ 25, 0, 5 ]
.
$ \end{tabular}  \\ 
\hline \end{tabular} \normalsize
\end{center}

\bibliographystyle{alpha}

\end{document}